\documentclass[a4paper]{amsart}
\usepackage{amssymb}
\usepackage{amscd}
\usepackage{amsthm}
\usepackage{amsmath}
\usepackage{mathtools}
\usepackage{latexsym}
\usepackage[all]{xy}
\usepackage[utf8]{inputenc}
\usepackage{enumitem}
\usepackage{hyperref}
\usepackage{color}
\usepackage{tikz-cd}
\usetikzlibrary{decorations.pathmorphing}

\DeclareMathOperator*{\ModR}{Mod-R}

\DeclareMathOperator*{\Hom}{Hom}
\DeclareMathOperator*{\End}{End}
\DeclareMathOperator*{\Ker}{Ker}

\DeclareMathOperator*{\Coker}{Coker}

\DeclareMathOperator*{\mSpec}{mSpec}

\DeclareMathOperator*{\Ann}{Ann}

\DeclareMathOperator*{\Add}{Add}

\DeclareMathOperator*{\Spec}{Spec}

\DeclareMathOperator*{\pd}{pd}

\DeclareMathOperator*{\ext}{Ext}
\DeclareMathOperator*{\tor}{Tor}

\setlist[enumerate]{label={(\roman*)}}%,parsep=\smallskipamount,itemsep=0pt}

\theoremstyle{plain}
\newtheorem{thm}{Theorem}
\newtheorem{cor}[thm]{Corollary}
\newtheorem{lem}[thm]{Lemma}
\newtheorem{prop}[thm]{Proposition}

\newtheorem{quest}[thm]{Question}

\theoremstyle{definition}
\newtheorem{definition}[thm]{Definition}
\newtheorem{example}[thm]{Example}

\theoremstyle{remark}

\newtheorem*{remark}{Remark}

\numberwithin{thm}{section}
\newif\ifcomments
\commentstrue

\definecolor{MichalH}{rgb}{0.9,0.1,0.1}

\ifcomments
\newcommand{\MichalH}[1]{{\color{MichalH}{#1}}}
\else
\newcommand{\MichalH}[1]{}
\fi
\begin{document}

\title{Divisibility classes are seldom closed under flat covers}
\author{Michal Hrbek}
\address{Institute of Mathematics CAS, \v{Z}itn\'{a} 25, 115 67 Prague, Czech Republic}
\address{
   Charles University \\
   Faculty of Mathematics and Physics\\
   Department of Algebra \\
   Sokolovsk\'a 83\\
   186 75 Praha 8\\
   Czech Republic}
\email{hrbek@math.cas.cz}
\date{\today}
\subjclass[2010]{Primary: 13G05, 13C60, 16E30, Secondary: 13B30, 13D30}
\keywords{commutative ring, tilting class, flat cover, almost perfect domain, divisible module}
\thanks{Research supported by GAČR 17-23112S and by RVO: 6798584}
\begin{abstract}
		It is well-known that a class of all modules, which are torsion-free with respect to a set of ideals, is closed under injective envelopes. In this paper, we consider a kind of a dual to this statement --- are the divisibility classes closed under flat covers? --- and argue that this is seldom the case. More precisely, we show that the class of all divisible modules over an integral domain $R$ is closed under flat covers if and only if $R$ is almost perfect. Also, we show that if the class of all $s$-divisible modules, where $s$ is a regular element of a commutative ring $R$, is closed under flat covers then the quotient ring $R/sR$ satisfies some rather restrictive properties. The question is motivated by the recent classification \cite{HS} of tilting classes over commutative rings.
\end{abstract}

\maketitle
%\tableofcontents

\section*{Introduction}
 	It is a basic fact of the theory of torsion pairs that a class of all modules, which are torsion-free with respect to a set of ideals, is closed under injective envelopes (and this indeed characterizes the torsion-free classes of torsion pairs which are generated by cyclic modules). Here we consider a sort of dual to this question --- are classes of all modules divisible by some set of ideals closed under flat covers? Our aim is to demonstrate that this is rather far from being the general case. Before that, we discuss in the rest of this introduction how this question is motivated by the tilting theory.

	Recently, a full classification of tilting classes (that is, the Ext-orthogonal classes to large tilting modules) over commutative rings was obtained, parametrizing the classes by certain subsets of the Zariski spectrum; first for noetherian rings in \cite{APST}, and then in \cite{HS} for the general setting. An interesting point is that the approach to the classification started in the dual setting (in the sense of elementary duality) of the cotilting classes first, and only afterwards the results were transferred back to tilting classes. The important property of the duals to tilting classes is that they are closed under injective envelopes (\cite[Proposition 5.5]{HS}). This led to the following result:
	\begin{thm}\emph{(\cite[Theorem 5.3]{HS})}\label{T:zero}
		Let $R$ be a commutative ring. Then the classes dual to $n$-tilting classes are precisely the classes of form
			$$\{M \in \ModR \mid \Hom{}_R(R/I,\Omega^{-i}M)=0 ~\forall i<n, ~\forall \text{\emph{ideals} $I$}, V(I) \subseteq X_i\},$$
			where $(X_0,X_1,\ldots,X_{n-1})$ is a characteristic sequence in $\Spec(R)$.
	\end{thm}
	The terminology is explained in the next section, but for now the main point is that these classes consist of all modules $M$ such that their $i$-th minimal cosyzygy $\Omega^{-i}M$ is torsion-free with respect to a family of (finitely generated) ideals. Afterwards, this can be further translated to a statement about Koszul homology, which dualizes well, yielding a classification of tilting classes as stated in Theorem~\ref{T:HS}. However, one might wonder whether we can play a similar game without using dualization. For that, a closure property analogous to the closure of the dual classes under injective envelopes would be required, and considering further that the tilting classes are expressed via the Tor functors (Theorem~\ref{T:HS}), this naturally leads to the following question:
	\begin{quest}\label{Q00}
		Are the tilting classes over a commutative ring closed under flat covers?
	\end{quest}
	An affirmative answer to Question~\ref{Q00} would yield an alternative description of the tilting classes analogous to Theorem~\ref{T:zero} --- instead of torsion-freeness of the minimal cosyzygies, the modules in the tilting class would be described by divisibility of their ``yokes'', that is, the cokernels in their minimal flat resolution. 
	
	In what follows, we show that the answer to Question~\ref{Q00} is a rather solid \textbf{NO}. In particular, we show that the tilting classes over an integral domain are closed under flat covers if and only if the domain is almost perfect (and thus, in particular, at most one-dimensional) --- Theorem~\ref{T:main}. By \cite{HS}, 1-tilting classes are precisely the classes of divisibility by finitely generated ideals with trivial annihilator. With respect to this, we cover a substantial family of tilting classes by proving that if a class of all modules divisible by a single regular element $s$ of some commutative noetherian ring $R$ is closed under flat covers, than the ring quotient $R/sR$ cannot contain a regular element --- Theroem~\ref{T:main2}, and also it cannot be a non-field integral domain under some size conditions.
\section{Preliminaries}
If not stated otherwise, our rings will be unital and commutative, and $\ModR$ will denote the category of all (right) $R$-modules. Let us gather some terminology and facts we will use freely in this paper:
\subsection{Filtrations} By a \emph{filtration} $M=\bigcup_{\alpha < \lambda}M_\alpha$ of an $R$-module $M$ we mean an increasing sequence 
$$0 = M_0 \subseteq M_1 \subseteq \cdots M_\alpha \subseteq M_{\alpha + 1} \subseteq \cdots$$
of submodules of $M$ indexed by an ordinal $\lambda$, such that its union is equal to the whole module $M$, and such that its \emph{continuous} --- that is, for any limit ordinal $\beta < \lambda$ we have $M_\beta = \bigcup_{\alpha < \beta} M_\alpha$. If $\mathcal{C}$ is a class of modules, a filtration $M=\bigcup_{\alpha<\lambda}M_\alpha$ is a $\mathcal{C}$\emph{-filtration} if $M_{\alpha+1}/M_\alpha$ is isomorphic to some module from $\mathcal{C}$ for all $\alpha<\lambda$. We say that $M$ is $\mathcal{C}$\emph{-filtered} (or \emph{filtered by modules from $\mathcal{C}$}), if $M$ admits a $\mathcal{C}$-filtration.
\subsection{Left and right approximations} Let $\mathcal{C}$ be a class of modules, and $M \in \ModR$. We say that a map $f: C \rightarrow M$ is a $\mathcal{C}$\emph{-precover}, if $C \in \mathcal{C}$, and for any $C' \in \mathcal{C}$ and any map $g: C' \rightarrow M$ there is a factorization map $h: C' \rightarrow C$ making the following triangle commute:
$$
\begin{tikzcd}
	C \arrow{r}{f} & M \\
	C'\arrow{ur}{g} \arrow[dashed]{u}{h}
\end{tikzcd}
$$
A $\mathcal{C}$-precover $f$ is a $\mathcal{C}$\emph{-cover} provided that any map $h \in \End_R(C)$, which makes the above diagram commute in the case when $f=g$, is necessarily an automorphism. We call a class $\mathcal{C}$ \emph{(pre)-covering} if any module $M$ admits a $\mathcal{C}$-(pre)cover. The notions of a $\mathcal{C}$\emph{-preenvelope} and $\mathcal{C}$\emph{-envelope} are defined dually, as well as the notion of a (pre)enveloping class. We recall that if a $\mathcal{C}$-cover ($\mathcal{C}$-envelope) exists, its domain (codomain) is uniquely determined up to an isomorphism (for details, see \cite[\S 5]{GT}).
\subsection{Cotorsion pairs} If $\mathcal{C}$ is a class of modules, we fix the following notation:
$$\mathcal{C}^\perp = \{M \in \ModR \mid \ext{}_R^1(C,M)=0 ~\forall C \in \mathcal{C}\},$$
$${}^\perp\mathcal{C} = \{M \in \ModR \mid \ext{}_R^1(M,C)=0 ~\forall C \in \mathcal{C}\},$$
with the obvious shorthand $M^\perp = \{M\}^\perp$ for a module $M \in \ModR$. A pair $(\mathcal{A},\mathcal{B})$ of subclasses of $\ModR$ is a \emph{cotorsion pair} provided that $\mathcal{B}=\mathcal{A}^\perp$ and $\mathcal{A}={}^\perp\mathcal{B}$. 

Given a class $\mathcal{C} \subseteq \ModR$, we say that an epimorphism $f: C \rightarrow M$ is a \emph{special $\mathcal{C}$-precover} if $\Ker(f) \in \mathcal{C}^\perp$. Dually, a monomorphism $f: M \rightarrow C$ is a \emph{special $\mathcal{C}$-preenvelope} if $\Coker(f) \in {}^\perp\mathcal{C}$. It is straightforward to check that any special $\mathcal{C}$-precover (special $\mathcal{C}$-preenvelope) is a $\mathcal{C}$-precover ($\mathcal{C}$-preenvelope). On the other hand, the Wakamatsu Lemma (\cite[Lemma 5.13]{GT}) shows that any epic $\mathcal{C}$-cover, and any monic $\mathcal{C}$-envelope, is special. Class $\mathcal{C}$ is called $\emph{special precovering}$ if any module from $\ModR$ admits a special $\mathcal{C}$-precover, and of course there is also the dual notion of a \emph{special preenveloping} class.

A cotorsion pair $(\mathcal{A},\mathcal{B})$ is \emph{complete} if $\mathcal{A}$ is special precovering (or equivalently, that $\mathcal{B}$ is special preenveloping, as proved by Salce \cite[Lemma 5.20]{GT}). The key result we will use is that any cotorsion pair generated by a set is complete, and admits a rather explicit description in terms of the generating set. A cotorsion pair $(\mathcal{A},\mathcal{B})$ is \emph{generated} by a set $\mathcal{S} \subseteq \ModR$ if $\mathcal{B}=\mathcal{S}^\perp$.

\begin{thm}\emph{(\cite[Theorem 6.11 and Corollary 6.14]{GT})}\label{T:abundant}
	Let $(\mathcal{A},\mathcal{B})$ be a cotorsion pair generated by a set $\mathcal{S} \subseteq \ModR$. Then:
\begin{itemize}
	\item the cotorsion pair $(\mathcal{A},\mathcal{B})$ is complete,
	\item the class $\mathcal{A}$ consists precisely of all direct summands of $\mathcal{S} \cup \{R\}$-filtered modules.
\end{itemize}
\end{thm}
\subsection{Tilting theory of commutative rings}
An $R$-module $T$ is \emph{n-tilting} for some $n \geq 0$ provided that it satisfies the following conditions:
\begin{enumerate}
	\item[(T1)] $\pd T \leq n$,
	\item[(T2)] $\ext_R^i(T,T^{(\kappa)})=0$ for all $i>0$ and all cardinals $\kappa$,
	\item[(T3)] there is an exact sequence
$$0 \rightarrow R \rightarrow T_0 \rightarrow T_1 \rightarrow \cdots \rightarrow T_n \rightarrow 0,$$
where $T_i \in \Add(T)$, the additive closure of $T$, for all $0 \leq i \leq n$.
\end{enumerate}
	The cotorsion pair $(\mathcal{A},\mathcal{B})$ generated by $T$ is called the \emph{n-tilting cotorsion pair}, and the class $\mathcal{B}$ the \emph{n-tilting class}. We suppress the index ($n$-) in the notation if we do not desire to specify the dimension bound on $T$. Two tilting modules are \emph{equivalent} if they induce the same tilting class. Even though the tilting modules can be infinitely generated, the tilting classes can always be described by ``finite data''. A module $S$ is called \emph{strongly finitely presented} if it admits a projective resolution consisting of finitely generated projective modules.

\begin{thm}\emph{(\cite[Theorem 13.46]{GT})}\label{T:finitetype}
		A class $\mathcal{B}$ is $n$-tilting if and only if there is a set $\mathcal{S}$ of strongly finitely presented modules of projective dimension at most $n$ such that $\mathcal{B}=\mathcal{S}^{\perp_\infty}=\{M \in \ModR \mid \ext_R^i(S,M)=0 ~\forall S \in \mathcal{S},~\forall i>0\}$.
\end{thm}

Over commutative rings, the tilting classes have been classified in terms of the spectrum of the ring. We say that a subset $X$ of the Zariski spectrum $\Spec(R)$ of the ring $R$ is \emph{Thomason}, provided that there is a set $\mathcal{I}$ of finitely generated ideals of $R$ such that $X=\bigcup_{I \in \mathcal{I}}V(I)$, where $V(I)=\{\mathfrak{p} \in \Spec(R) \mid I \subseteq \mathfrak{p}\}$ is the basic Zariski-closed set defined on $I$. A sequence $(X_0,X_1,\ldots,X_{n-1})$ of Thomason subsets of $\Spec(R)$ is called \emph{characteristic} if the following conditions are satisfied:
\begin{enumerate}
	\item[(i)] $X_0 \supseteq X_1 \supseteq \cdots \supseteq X_{n-1}$,
	\item[(ii)] $\ext_R^i(R/I,R)=0$ for any $i<n$ and any finitely generated ideal $I$ such that $V(I) \subseteq X_i$.
\end{enumerate}

\begin{thm}\emph{(\cite[Theorem 6.4]{HS})}\label{T:HS}
	Let $R$ be a commutative ring and $n>0$. There is a 1-1 correspondence
	$$\left \{ \begin{tabular}{ccc} \text{ characteristic sequences} \\ \text{ $(X_0,X_1,\ldots,X_{n-1})$ in $\Spec(R)$} \end{tabular}\right \}  \leftrightarrow  \left \{ \begin{tabular}{ccc} \text{ $n$-tilting classes $\mathcal{T}$ } \\ \text{ in $\ModR$} \end{tabular}\right \},$$
given by the following assignment:
$$(X_0,X_1,\ldots,X_{n-1}) \rightarrow \{M \in \ModR \mid \tor{}_i^R(R/I,M)=0 ~\forall i<n, ~\forall V(I) \subseteq X_i\}.$$
\end{thm}

\begin{remark}
	In this paper, we will mostly deal with 1-tilting classes, for which Theorem~\ref{T:HS} gives a very explicit description --- they are precisely the classes $\mathcal{D}_\mathcal{I}=\{M \in \ModR \mid M=IM ~\forall I \in \mathcal{I}\}$ of all modules divisible by a set $\mathcal{I}$ of finitely generated ideals $I$ satisfying $\Hom_R(R/I,R) = 0$.

		Also, it follows from \cite{HS} that we do not need to consider all ideals $I$ with $V(I) \subseteq X_i$ in Theroem~\ref{T:HS} --- it is enough to test a set $\mathcal{I}_i$ of finitely generated ideals such that $X_i=\bigcup_{I \in \mathcal{I}_i}V(I)$.
\end{remark}

\section{Divisibility classes and flat covers}
	Let $\mathcal{F}$ be a covering class. Say that a class of modules $\mathcal{C}$ is \emph{closed under $\mathcal{F}$-covers} if for any $\mathcal{F}$-cover $h: F \rightarrow M$ with $M \in \mathcal{C}$ we have $F \in \mathcal{C}$. The following lemma is simple, but key to our investigation, as it allows to study properties of divisible preenvelopes, rather than flat covers, which are often hard to compute.
	\begin{lem}
			\label{PL00}
			Let $\mathcal{C}$ be a preenveloping class closed under direct summands and $\mathcal{F}$ a covering class. Then $\mathcal{C}$ is closed under $\mathcal{F}$-covers if and only if any module $F \in \mathcal{F}$ admits a $\mathcal{C}$-preenvelope $f: F \rightarrow L$ such that $L \in \mathcal{F}$.
	\end{lem}
	\begin{proof}
			($\Rightarrow$) Let $F \in \mathcal{F}$ and $f: F \rightarrow C$ a $\mathcal{C}$-preenvelope. By the assumption, the domain of the $\mathcal{F}$-cover $h: F' \rightarrow C$ of $C$ belongs to $\mathcal{C}$. Then there is a map $g: F \rightarrow F'$ such that $f=hg$. We claim that $g$ is a $\mathcal{C}$-preenvelope. Indeed, if $l: F \rightarrow C'$ is map with $C' \in \mathcal{C}$, then there is factorization map $g': C \rightarrow C'$ such that $g'f=l$. Then $(g'h)g=g'f=l$, and thus $g$ is a $\mathcal{C}$-preenvelope with codomain in $\mathcal{F}$, as desired.

			($\Leftarrow$) Let $C \in \mathcal{C}$ and consider its $\mathcal{F}$-cover $h: F \rightarrow C$. Using the assumption, there is a $\mathcal{C}$-preenvelope of $F$, say $f: F \rightarrow L$, such that $L \in \mathcal{F}$. Since $f$ is a $\mathcal{C}$-preenvelope, there is a factorization $g: L \rightarrow C$ with $h=gf$. Because $h$ is an $\mathcal{F}$-precover, there is a factorization $l: L \rightarrow F$ such that $g=hl$. The situation is encaptured in the following diagram:
			\begin{center}
$$
	\begin{tikzcd}
		F \arrow{r}{h}\arrow{d}{f} & C \\
		L \arrow[bend left=40,squiggly]{u}{l} \arrow[dashed]{ur}[below]{g}
	\end{tikzcd}
$$
			\end{center}
			Hence, $hlf=gf=h$, and thus $lf$ is an automorphism of $F$. Therefore, $f$ is a split monomorphism, and whence $F \in \mathcal{C}$, since $\mathcal{C}$ is closed under direct summands.
	\end{proof}
By a flat cover, we mean a $\mathcal{F}_0$-cover, where $\mathcal{F}_0$ denotes the class of all flat $R$-modules. Recall that $\mathcal{F}_0$ is a covering class by the celebrated result due to Bican, El Bashir, and Enochs \cite{BBE}. 
	\begin{cor}\label{PC00}
			A tilting class $\mathcal{T}$ is closed under flat covers if and only if any flat $R$-module has a $\mathcal{T}$-preenvelope which is flat.\footnote{Here, ``$\mathcal{T}$-preenvelope'' is a shorthand for ``codomain of the $\mathcal{T}$-preenvelope'', which we will use throughout the paper. Also, note that ``flat $\mathcal{T}$-preenvelope'' is not to be confused with $(\mathcal{T} \cap \mathcal{F}_0)$-preenvelope.}
	\end{cor}
\begin{proof}
	Straightforward from Lemma~\ref{PL00}, as tilting classes are preenveloping by Theorem~\ref{T:abundant} and clearly closed under direct summands.
\end{proof}
\subsection{Divisible modules over integral domains}
	We start with the case of an integral domain $R$ and the class $\mathcal{D}=\{M \in \ModR \mid rM=M \: \forall 0 \neq r \in R\}$ of all classical divisible $R$-modules. Then $\mathcal{D}=(\bigoplus_{0 \neq r \in R}R/rR)^{\perp_\infty}$ is a 1-tilting class in $\ModR$ via Theorem~\ref{T:finitetype}. We prove that $\mathcal{D}$ is closed under flat covers if and only if $R$ is an almost perfect domain, and in this case, all tilting classes are closed under flat covers. 	
	
	During the rest of the section, let $Q$ always denote the field of quotients of integral domain $R$. 
	\begin{lem}\label{PL20}
		Suppose that $F$ is a flat $R$-module such that $F$ admits a flat $\mathcal{D}$-preenvelope. Than the canonical map $F \rightarrow F \otimes_R Q$ is a $\mathcal{D}$-envelope of $F$.
	\end{lem}
	\begin{proof}
			Let $f: F \rightarrow L$ be a $\mathcal{D}$-preenvelope of $F$ such that $L$ is flat, and let $\iota: F \rightarrow F \otimes_R Q$ denote the canonical map. Since $L$ is flat (and thus torsion-free), and divisible, we have a natural isomorphism $L \otimes_R Q \simeq L$. Then tensoring $f$ by $Q$ shows that $f$ factors through $\iota$. This already implies that $\iota$ is a $\mathcal{D}$-preenvelope of $F$. Because $\iota$ is an essential monomorphism, it is easy to deduce that it has to be a $\mathcal{D}$-envelope. Indeed, any endomorphism $f \in \End_R(Q)$ such that $f\iota=\iota$ has to be a monomorphism by the essentiality of the monomorphism $\iota$. Since $Q$ is injective, $f$ is a split monomorphism, and thus $f$ is an isomorphism again by essentiality of $\iota$.
	\end{proof}
	Recall that an integral domain $R$ is called \emph{Matlis}, if $\pd Q \leq 1$. 
	\begin{lem}
			\label{PL02}
			Let $R$ be an integral domain. If $\mathcal{D}$ is closed under flat covers, then $R$ is Matlis.
	\end{lem}
	\begin{proof}
			If $\mathcal{D}$ is closed under flat covers, then Lemma~\ref{PL00}, followed by Lemma~\ref{PL20}, yields that the canonical map $\iota: R^{(\varkappa)} \rightarrow Q^{(\varkappa)}$ is a $\mathcal{D}$-preenvelope for all cardinals $\varkappa$. But then $Q$ generates $\mathcal{D}$, which implies $\pd Q \leq 1$ by \cite[Theorem 1.1]{AHT}.
	\end{proof}

	\begin{lem}
			\label{PL01}
		Let $R$ be a Matlis domain. If there is a flat module $F$ such that $\pd F >1$, then the class $\mathcal{D}$ of all divisible $R$-modules is not closed under flat covers.
	\end{lem}
	\begin{proof}
			Using Lemma~\ref{PL00}, it is enough to show that $F$ has no flat $\mathcal{D}$-preenvelope. Towards a contradiction, suppose otherwise. Therefore, by Lemma~\ref{PL20}, the canonical map $\iota: F \rightarrow F \otimes_R Q$ is a $\mathcal{D}$-envelope. By the Wakamatsu Lemma (\cite[Lemma 5.13]{GT}), $\iota$ is a special $\mathcal{D}$-preenvelope, and thus $\pd \Coker(\iota) \leq 1$. But since $F \otimes_R Q$ is isomorphic to a direct sum of copies of $Q$, we conclude from $R$ being a Matlis domain that $\pd (F \otimes_R Q) \leq 1$, and thus $\pd F \leq 1$, a contradiction.
	\end{proof}
	An integral domain is said to be \emph{almost perfect} if any proper factor of $R$ is a perfect ring (i.e., a ring such that any flat module is projective). These domains were introduced by Bazzoni and Salce and have many equivalent characterizations (see e.g. \cite{Salce}). We list just a few of them, which will be useful in what follows. Recall that an integral domain $R$ is \emph{h-local} if the following two conditions hold:
	\begin{enumerate}
		\item $R$ has finite character, that is, each non-zero element belongs to only finitely many maximal ideals of $R$, and
		\item every non-zero prime ideal of $R$ is contained in precisely one maximal ideal of $R$.
	\end{enumerate}
	\begin{thm}{\emph{\cite[Main Theorem]{Salce}}}\label{T:apd}
		For an integral domain $R$, the following conditions are equivalent:
			\begin{enumerate}
				\item $R$ is almost perfect,
				\item $R$ is $h$-local, and every localization of $R$ is almost perfect,
				\item every $R$-module of weak dimension $\leq 1$ has projective dimension $\leq 1$.
			\end{enumerate}
	\end{thm}
	We also note some properties of almost perfect domains:
	\begin{lem}\label{PL15}
			Let $R$ be an almost perfect domain. Then $R$ is Matlis, and has Krull dimension at most 1. A noetherian domain of Krull dimension at most 1 is almost perfect.
	\end{lem}
	\begin{proof}
			The first two properties are proved in \cite[Proposition 3.5]{Salce}. The last claim is \cite[Proposition 5.1]{Salce}.
	\end{proof}
	Before proceeding with studying flat covers, we prove some properties of $h$-local domains of Krull dimension 1 concerning Thomason sets and the tilting modules. By $\mSpec(R)$ we denote the subset of $\Spec(R)$ consisting of all maximal ideals. Reader interested only in flat cover closure only needs to be concerned with Proposition~\ref{PP01}, and then skip to Lemma~\ref{PL03}. Also, given a maximal ideal $\mathfrak{m}$ we denote by $R_\mathfrak{m}$ the localization of $R$ in $\mathfrak{m}$, and for any subset $X$ of $\mSpec(R)$ let $R_X = \bigcap_{\mathfrak{m} \in X} R_\mathfrak{m}$. If $\mathfrak{m} \in \mSpec(R)$, we denote for convenience by $[\mathfrak{m}]=\mSpec(R) \setminus \{\mathfrak{m}\}$ the complement of $\mathfrak{m}$ in $\mSpec(R)$. Given a module $M$, we adopt the shorthand notation $M_\mathfrak{m} = M \otimes_R R_\mathfrak{m}$.
	\begin{prop}\label{PP01}
		Let $R$ be an integral domain of Krull dimension 1. Then any subset of $\mSpec(R)$ is Thomason if and only if $R$ is $h$-local.
	\end{prop}
	\begin{proof}
					Let $R$ be of Krull dimension 1. First, assume that $R$ is $h$-local and let us show that any set of maximal ideals is Thomason. It is clearly enough to show that any singleton set $\{\mathfrak{m}\} \subseteq \mSpec(R)$ is Thomason, that is, find a finitely generated ideal $I$ such that $V(I)=\{\mathfrak{m}\}$. Pick any non-zero element $x_0 \in \mathfrak{m}$. Since $R$ is $h$-local, $V(x_0)$ is a finite set containing $\mathfrak{m}$. Now for any $\mathfrak{n} \in V(x_0)$ not equal to $\mathfrak{m}$, there is an element $x_\mathfrak{n} \in \mathfrak{m}$ such that $x_\mathfrak{n} \not\in \mathfrak{n}$. Hence, the ideal $I$ generated by $x_0$ and all the $x_\mathfrak{n}$'s is a finitely generated ideal satisfying $V(I)=\{\mathfrak{m}\}$.

							Conversely, suppose that for any maximal ideal $\mathfrak{m}$ there is an ideal $I$ generated by $x_1,\ldots,x_n$ such that $V(I) = \{\mathfrak{m}\}$, and let us show that $R$ is $h$-local. By \cite[IV.3. Theorem 3.7]{NN}, it is enough to show that for each $\mathfrak{m} \in \mSpec(R)$, the module $A=R_\mathfrak{m} \otimes_R R_{[\mathfrak{m}]}$ is divisible (and thus, isomorphic to $Q$). There is a disjoint partition of $[\mathfrak{m}]=X_1 \cup X_2\ \cup \ldots \cup X_n$ such that for each $i=1,\ldots,n$, $x_i$ is not contained in any maximal ideal from the set $X_i$. It follows that $R_{X_i}=\bigcap_{\mathfrak{n} \in X_i}R_\mathfrak{n}$ is divisible by $x_i$; indeed, $\frac{1}{x_i} \in R_\mathfrak{n}$ for each $\mathfrak{n} \in X_i$. As localizing commutes with finite intersections, we have that $A=\bigcap_{1 \leq i \leq n} A_i$, where $A_i=R_\mathfrak{m} \otimes_R R_{X_i}$. Let $S_i=\{x_i^n \mid n\geq 0\}$ be the multiplicative set generated by $x_i$. Then $A_iS_i^{-1}=(R_\mathfrak{m} S_i^{-1}) \otimes_R R_{X_i}$ is a divisible $R$-module, because $x_i \in \mathfrak{m}$, and $R_\mathfrak{m}$ is a 1-dimensional local domain. Since $x_i \not \in \bigcup X_i$, the module $A_i$ is $S_i$-divisible, that is $A_i=x_iA_i$. On the other hand, we showed that the localization $A_iS_i^{-1}$ is divisible (by any non-zero element of $R$). It follows that $A_i$ is divisible for each $i$. Therefore, $A=\bigcap_{1 \leq i \leq n}A_i$ is an intersection of torsion-free divisible modules, which is always divisible, as desired.
	\end{proof}
	\begin{lem}\label{PL89}
		Let $R$ be an integral domain of Krull dimension at most 1. Then any tilting class is 1-tilting.
	\end{lem}
	\begin{proof}
			With respect to Theorem~\ref{T:HS}, it is enough to show that $\ext_R^1(R/I,R) \neq 0$ for any non-zero finitely generated ideal $I$ (in the noetherian setting, this follows directly from the classical grade theory, see e.g. \cite[Proposition 1.2.14]{BH}). Applying $\Hom_R(R/I,-)$ to the exact sequence $0 \rightarrow R \rightarrow Q \rightarrow Q/R \rightarrow 0$ yields $\ext_R^1(R/I,R) \simeq \Hom_R(R/I,Q/R)$. Since $R/I$ is finitely presented, the vanishing of the latter Hom group is equivalent to $\Hom_{R_\mathfrak{m}}(R_\mathfrak{m}/I_\mathfrak{m},Q/R_\mathfrak{m})=0$ for all $\mathfrak{m} \in \mSpec(R)$. Suppose that $\mathfrak{m} \in V(I)$. Because $I_\mathfrak{m}$ is a proper non-zero ideal of $R_\mathfrak{m}$, this implies that $Q/R_\mathfrak{m}$ belongs to a torsion-free class $\mathcal{F}$ of a hereditary torsion pair $(\mathcal{T},\mathcal{F})$ in $\operatorname{Mod-}R_\mathfrak{m}$, where $\mathcal{F}$ is neither $\{0\}$ nor the whole $\operatorname{Mod-}R_\mathfrak{m}$. But by the locality and 1-dimensionality of $R_\mathfrak{m}$, the only possibility is that $\mathcal{F}$ is the class of classical torsion-free modules, and therefore $Q/R_\mathfrak{m}$ is torsion-free (for details on the relevant torsion pair theory aspects, see e.g. \cite[Proposition 2.11]{HS}). As the inclusion $R_\mathfrak{m} \subseteq Q$ is essential, this implies $R_\mathfrak{m}=Q$ for any $\mathfrak{m} \in V(I)$. This can happen only in the situation $Q=R$, and then the whole claim is trivially valid.
	\end{proof}
	\begin{definition}
			Let $R$ be a Matlis domain with quotient field $Q$. Denote by $\pi: Q \rightarrow Q/R$ the canonical projection, and let $A$ be any direct summand of $Q/R$. Denote by $\pi^{-1}[A]$ the full preimage of $A$ under $\pi$. Then the module $T_A=A \oplus \pi^{-1}[A]$ is the \emph{Bass tilting module} associated to $A$.
	\end{definition}
	We quickly check that $T_A$ is indeed a tilting module. Conditions (T1) and (T3) follows directly from the definition; let us check the condition (T2), that is, $\ext_R^{1}(T_A,T_A^{(\varkappa)})=0$ for any cardinal $\varkappa$. Recall, denoting by $\mathcal{P}_1$ the class of all $R$-modules of projective dimension $\leq 1$, that $(\mathcal{P}_1,\mathcal{D})$ is a cotorsion pair (see \cite[Proposition 6.3]{BH2}). Since $\pd T_A \leq 1$, and $A$ is divisible, we can already infer that $\ext_R^1(T_A,A^{(\varkappa)})=0$. Denote $S=\pi^{-1}[A]$. It remains to check that $\ext_R^1(T_A,S^{(\varkappa)})=0$. Consider the exact sequence
	$$0 \rightarrow S \rightarrow Q \rightarrow Q/S \rightarrow 0.$$
	Applying $\Hom_R(A,-)$ yields an isomorphism $\Hom_R(A,Q/S) \simeq \ext_R^1(A,S)$, because $A$ is a torsion module of projective dimension $\leq 1$. Note that since $S/R \simeq A$, then $Q/S \simeq B$, where $Q/R = A \oplus B$. It follows from \cite[Lemma 4.2(b),(c)]{NN}, that either $A_\mathfrak{m}=Q/R_\mathfrak{m}$ or $B_\mathfrak{m}=Q/R_\mathfrak{m}$ for any $\mathfrak{m} \in \mSpec(R)$. Therefore, for any $\mathfrak{m} \in \mSpec(R)$, either $A_\mathfrak{m} = 0$ or $B_\mathfrak{m}=0$. Consequently, $\Hom_R(A,B)=0$, as $\operatorname{Im}(f) \otimes_R R_\mathfrak{m} = \operatorname{Im}(f \otimes_R R_\mathfrak{m})=0$ for any $f\in \Hom_R(A,B)$, and any $\mathfrak{m} \in \Spec(R)$. Therefore, $0 = \Hom_R(A,B^{(\varkappa)}) \simeq \ext_R^1(A,S^{(\varkappa)})$. Using the exact sequence
	$$0 \rightarrow R \rightarrow S \rightarrow A \rightarrow 0$$
	it is clear that also $\ext_R^1(S,S^{(\varkappa)})=0$, and thus we can finally conclude that $\ext_R^1(T_A,T_A^{(\varkappa)})=0$.
	\begin{prop}
		Let $R$ be a Matlis domain. Then any tilting module is equivalent to a Bass tilting module if and only if $R$ is $h$-local and of Krull dimension at most 1.
	\end{prop}
	\begin{proof}
			Let $R$ be a 1-dimensional $h$-local domain. By Lemma~\ref{PL89}, any tilting class is 1-tilting. Let $\mathcal{T}$ be a 1-tilting class, and let $X \subseteq \mSpec(R)$ be the Thomason set corresponding to $\mathcal{T}$ in the sense of Theorem~\ref{T:HS}. Explicitly, with respect to Proposition~\ref{PP01}, there is a finitely generated ideal $I_\mathfrak{m}$ such that $V(I_\mathfrak{m})= \{\mathfrak{m}\}$, and $\mathcal{T}=\{M \in \ModR \mid M=I_\mathfrak{m}M \: \forall \mathfrak{m} \in X\}$. Since $R$ is $h$-local, there is by \cite[Theorem 3.7]{NN} a natural isomorphism $Q/R \simeq \bigoplus_{\mathfrak{m} \in \mSpec(R)}Q/R_\mathfrak{m} \simeq \bigoplus_{\mathfrak{m} \in \mSpec(R)}R_{[\mathfrak{m}]}/R$. We let $A_X=\bigoplus_{\mathfrak{m} \in X}Q/R_\mathfrak{m}$ and let $T_X=A_X \oplus \pi^{-1}[A_X]$ be the Bass tilting module corresponding to $A_X$. It is easy to see that $T_X$ is divisible by a finitely generated ideal $I$ if and only if $V(I) \subseteq X$ (as $\pi^{-1}[A_X]=R_{\mSpec(R) \setminus X}$). Since $T_X$ is a 1-tilting module, it has to generate the tilting class $\mathcal{T}$.

			Suppose, on the other hand, that any tilting class in $\ModR$ is generated by a Bass tilting module. First, let us show that the Krull dimension of $R$ is at most 1. Indeed, if $\mathfrak{p}$ is a non-zero and non-maximal prime ideal, consider the tilting class $\mathcal{T}_S$ of all $S$-divisible modules, where $S=R \setminus \mathfrak{p}$. Suppose $\mathcal{T}_S$ is generated by a Bass tilting module $T$. Localizing at any maximal ideal $\mathfrak{m}$ containing $\mathfrak{p}$, we can assume without loss of generality that $R$ is local --- clearly a localization of a Bass tilting module at $\mathfrak{m}$ is a Bass tilting module over $R_\mathfrak{m}$, and by \cite[Proposition 13.50]{GT}, the tilting class generated by $T_\mathfrak{m}$ consists of all $S_\mathfrak{m}=(R_\mathfrak{m} \setminus \mathfrak{p}_\mathfrak{m})$-divisible modules. Since $R$ is now a assumed to be local, $Q/R$ is indecomposable, and thus there are only two Bass tilting modules over $R$ --- the trivial one generating $\ModR$, and then $Q \oplus Q/R$ generating $\mathcal{D}$. As none of these generates $\mathcal{T}_S$, we established a contradiction, and thus $R$ is indeed 1-dimensional.

			Finally we prove that $R$ is $h$-local. Given any maximal ideal $\mathfrak{m}$, the cofinite subset $[\mathfrak{m}]$ of $\mSpec(R)$ is easily seen to be Thomason. Therefore, there is a tilting class $\mathcal{T}$ corresponding to $[\mathfrak{m}]$ as in the first part of this proof. By the hypothesis, $\mathcal{T}$ is generated by a Bass tilting module $T_A=A \oplus \pi^{-1}[A]$ for some direct summand $A$ of $Q/R$. By \cite[\S IV, Lemma 4.2b]{NN}, necessarily $A \simeq R_\mathfrak{m}/R$ and $Q/R$ is naturally isomorphic to $R_\mathfrak{m}/R \oplus R_{[\mathfrak{m}]}/R$. Then the complement $B=R_{[\mathfrak{m}]}/R$ is non-zero, and if $T_B=B \oplus \pi^{-1}[B]$ is the Bass tilting module corresponding to $B$, the tilting class is not equal to $\ModR$. Furthermore, any ideal $I$ such that $T_B=IT_B$ is not contained in any maximal ideal other than $\mathfrak{m}$. Then the Thomason set corresponding to $T_B$ needs to be the singleton $\{\mathfrak{m}\}$. We proved that any subset of $\mSpec(R)$ is Thomason, and therefore $R$ is $h$-local by Proposition~\ref{PP01}.
	\end{proof}
	\begin{lem}
			\label{PL03}
			Let $R$ be an integral domain. If $\mathcal{D}$ is closed under flat covers, then $R$ is almost perfect.
	\end{lem}
	\begin{proof}
		By combining Lemmas~\ref{PL01} and \ref{PL02}, we know that the hypothesis implies that all flat $R$-modules have projective dimension at most 1. Let $M$ be of flat dimension 1, that is, we have an exact sequence 
		$$0 \rightarrow F_1 \rightarrow F_0 \rightarrow M \rightarrow 0$$ 
		with $F_0,F_1$ flat. Let $Z$ denote the cokernel of the composition $F_1 \rightarrow F_0 \rightarrow F_0 \otimes_R Q$. Consider the exact sequence
		$$0 \rightarrow (F_1 \otimes_R Q)/F_1 \rightarrow Z \rightarrow M \otimes_R Q \rightarrow 0.$$
			The leftmost term has projective dimension $\leq 1$, because it is the cokernel of a special $\mathcal{D}$-preenvelope (Lemma~\ref{PL20}), while the rightmost term has projective dimension $\leq 1$ because $R$ is Matlis, and $M \otimes_R Q$ is a direct sum of copies of $Q$. Hence, $\pd Z \leq 1$. Now, let us focus our attention on the exact sequence
		$$0 \rightarrow M \rightarrow Z \rightarrow (F_0 \otimes_R Q)/F_0 \rightarrow 0,$$
		induced by the map $M \simeq F_0/F_1 \subseteq Z$.	
			We already know that $\pd Z \leq 1$. The rightmost term of the sequence is again the cokernel of a special $\mathcal{D}$-preenvelope, whence it is of projective dimension at most 1 too. Therefore, $\pd M \leq 1$ as desired. We proved that $R$-modules of weak dimension $\leq 1$ coincide with the $R$-modules of projective dimension $\leq 1$, proving that $R$ is almost perfect by Theorem~\ref{T:apd}.
	\end{proof}
	In the final step, we prove that all tilting classes over an almost perfect domain are closed under flat covers. Before that, we recall some required basics of the fractional ideal theory. Given an ideal $I$, let $I^{-1} = \{r \in Q \mid rI \subseteq R\}$. Ideal $I$ is called \emph{invertible} if $II^{-1}=R$. For any $n > 0$, we set $I^{-n}=(I^n)^{-1}$ and $I^0=R$.
	\begin{lem}\label{PL19}
		Let $R$ be an integral domain and $I$ a non-zero ideal. Then:
			\begin{enumerate}
				\item[(i)] $I$ is projective if and only if it is invertible, and in this case, $I$ is finitely generated.
				\item[(ii)] If $I$ is invertible, the quotient of fractional ideals $I^{-(n+1)}/I^{-n}$ is a projective generator in $\mbox{Mod-$R/I$}$ for any $n \geq 0$.
				\item[(iii)] If $I$ is invertible, then $\ext_R^1(P,D)=0$ for any projective $R/I$-module $P$ and any $D$ such that $D=ID$.
			\end{enumerate}
	\end{lem}
	\begin{proof}
			\begin{enumerate}
				\item[(i)] See \cite[Proposition 7.2 and Lemma 7.1]{Pass}.
				\item[(ii)] Consider the tensor product $I^{-(n+1)} \otimes_R R/I$. By right exactness of tensoring, we have natural isomorphisms
						$$I^{-(n+1)} \otimes_R R/I \simeq (I^{-(n+1)} \otimes_R R)/(I^{-(n+1)} \otimes_R I) \simeq I^{-(n+1)}R/I^{-(n+1)}I.$$
							If $I$ is invertible, then $I^n$ is invertible too, and $I^{-(n+1)}=(I^{-1})^{n+1}$. This shows that $I^{-(n+1)} \otimes_R R/I \simeq I^{-(n+1)}/I^{-n}$. Therefore, $I^{-(n+1)}/I^{-n}$ is an $R/I$-module, and since $I^{-(n+1)}$ is a projective $R$-module, $I^{-(n+1)}/I^{-n}$ is a projective $R/I$-module. Since $R$ is a domain, $I^{-(n+1)}$ is a projective generator (as the trace ideal is a pure ideal of $R$), and thus $I^{-(n+1)} \otimes_R R/I$ is also a projective generator in $\mbox{Mod-$R/I$}$.
					\item[(iii)] Because $I$ is projective by $(i)$, we have that $I \rightarrow R \rightarrow R/I \rightarrow 0$ is a projective presentation of $R/I$. Therefore, $I^{-1}/R$ is an Auslander-Bridger transpose of $R/I$, and thus $(I^{-1}/R)^{\perp_1}=\{D \in \ModR \mid D=ID\}$ (for details, see \cite[\S 3 \& Theorem 3.16]{Hrb}). By $(ii)$, $I^{-1}/R$ is a projective generator of $\mbox{Mod-$R/I$}$, thus $\Add(R/I)=\Add(I^{-1}/R)$, and therefore $\ext_R^1(P,D)=0$ for any $P \in \Add(R/I)$ and any $D \in \ModR$ such that $D=ID$.
			\end{enumerate}
	\end{proof}
	\begin{lem}\label{L:fre}
		Let $R$ be an integral domain and $I$ an invertible ideal. Denote by $Q_I$ the directed union of fractional ideals $\bigcup_{n>0} I^{-n} \subseteq Q$. Then the inclusion map $\varphi: R \rightarrow Q_I$ is a flat ring epimorphism, and also a $\mathcal{D}_I$-envelope.

			Furthermore, suppose that $F$ is a flat module which admits a flat $\mathcal{D}_I$-preenvelope. Then the map $F \rightarrow F \otimes_R Q_I$ is a $\mathcal{D}_I$-envelope.
	\end{lem}
	\begin{proof}
			That $Q_I$ is a flat ring epimorphism follows from \cite[\S IX, Proposition 2.2 and 2.4]{SS}. Also, $Q_I \in \mathcal{D}_I$ by \cite[\S XI, Proposition 3.4]{SS}. The cokernel of $\varphi$ is isomorphic to a directed union $\bigcup_{n>0} I^{-n}/R$, which is in turn a module filtered by modules of form $I^{-(n+1)}/I^{-n}, n \geq 0$. These modules are projective $R/I$-modules by Lemma~\ref{PL19}, and thus $\Coker(\varphi)$ belongs to ${}^\perp\mathcal{D}_I$ by Lemma~\ref{PL19}(iii). Therefore, $\varphi$ is a special $\mathcal{D}_I$-preenvelope. That $\varphi$ is actually a $\mathcal{D}_I$-envelope follows from the fact that $\varphi: R \rightarrow Q_I$ is a ring epimorphism, and therefore any $R$-module map $f: Q_I \rightarrow Q_I$ fixing the unit element is an isomorphism.

			Now we prove the ``furthermore'' part. Let $f: F \rightarrow L$ be a $\mathcal{D}$-preenvelope such that $L$ is flat. Since $L$ is flat (and thus, in particular, torsion-free) and $I$-divisible, we infer from \cite[Proposition 3.4]{SS} that $L$ has a natural structure of a $Q_I$-module. Then tensoring $f$ with $Q_I$ yields a sequence of embeddings $F \subseteq F \otimes_R Q_I \subseteq L$, such that the composition of both inclusions is $f$. From this it is easy to infer that the embedding $\iota: F \rightarrow F \otimes_R Q_I$ is a $\mathcal{D}_I$-preenvelope. Let $g \in \End_R(F \otimes_R Q_I)$ be a map such that $g\iota=\iota$. Since $\iota$ is an essential embedding, $g$ is necessarily a monomorphism. Denote by $F'$ the image of $g$ in $F \otimes_R Q_I$. Then $F'$ is a $Q_I$-module and we have a chain of inclusions $F \subseteq F' \subseteq F \otimes_R Q_I$. Tensoring this by $Q_I$ yields $F' \otimes_R Q_I = F \otimes_R Q_I$, implying that $F'=F \otimes_R Q_I$. Therefore, $g$ is an automorphism, and so $\iota: F \rightarrow F \otimes_R Q_I$ is a $\mathcal{D}_I$-envelope, as claimed.
	\end{proof}
	\begin{prop}\label{PP18}
			Let $R$ be an integral domain and $\mathcal{I}$ a set of invertible ideals. If $R/I$ is a perfect ring for all $I \in \mathcal{I}$, then the class $\mathcal{D}_\mathcal{I}=\{M \in \ModR \mid M=IM\ ~\forall I \in \mathcal{I}\}$ of all $\mathcal{I}$-divisible modules is closed under flat covers.
	\end{prop}
	\begin{proof}
			First, we claim that we can without any loss of generality assume that $\mathcal{I}=\{I\}$ is a singleton. Indeed, if we prove the claim for all singleton sets, then the general statement follows from the simple observation that $\mathcal{D}_\mathcal{I}=\bigcap_{I \in \mathcal{I}} \mathcal{D}_I$.

			Let $Q_I=\bigcup_{n \in \mathbb{N}}I^{-n}$ be the ring epimorphic extension of $R$ from Lemma~\ref{L:fre}, and let $Y=\Coker(\varphi)$. Fix a flat $R$-module $F$, and consider the exact sequence:
			$$0 \rightarrow F \rightarrow F \otimes_R Q_I \rightarrow X \rightarrow 0.$$
			Any presentation of $F$ as a direct limit of free modules yields that $X$ is a direct limit of copies of $Y$, so there is a pure exact sequence:
			$$\epsilon: 0 \rightarrow K \xrightarrow{*} Y^{(\varkappa)} \rightarrow X \rightarrow 0.$$
			Denote by $Y_n=\Hom_R(R/I^n,Y)$ the $I^n$-socle of $Y$. Applying $\Hom_R(R/I^n,-)$ yields an exact sequence
			$$\epsilon': 0 \rightarrow K_n \rightarrow (Y_n)^{(\varkappa)} \rightarrow X_n \rightarrow 0,$$
			where $K_n$ and $X_n$ are the $I^n$-socles of $K$ and $X$, accordingly. This exact sequence is pure in $R/I^n$-Mod. Indeed, if $G$ is a finitely presented $R/I^n$-module, the sequence $\Hom_{R/I^n}(G,\epsilon')$ is naturally isomorphic to $\Hom_{R/I^n}(G,\Hom_R(R/I^n,\epsilon))$. We have a natural isomorphism (of complexes)
			$$\Hom{}_{R/I^n}(G,\Hom{}_R(R/I^n,\epsilon)) \simeq \Hom{}_{R}(G,\epsilon).$$
			Since $I^n$ is finitely generated, $G$ is a finitely presented $R$-module, and because $\epsilon$ is an exact sequence in $\ModR$, the resulting sequence is exact.

			Now apply the tensor functor $R/I \otimes_{R_{I^n}} - $ onto the pure exact sequence $\epsilon'$, which yields a pure exact sequence
			$$0 \rightarrow K_n/IK_n \xrightarrow{*} (Y_n/IY_{n})^{(\varkappa)} \rightarrow X_n/IX_{n} \rightarrow 0.$$
			Since $IY_n = Y_{n-1}$, it follows that $X_n/IX_n=X_n/X_{n-1}$. As this sequence is pure, and $Y_n/Y_{n-1}$ is a projective $R/I$-module by Lemma~\ref{PL19}, it follows that $X_n/X_{n-1}$ is a flat $R/I$-module. But $R/I$ is a perfect ring, and thus $X_n/X_{n-1}$ is actually a projective $R/I$-module, and therefore belongs to ${}^{\perp} \mathcal{D}_I$ by Lemma~\ref{PL19}. As $X$ is filtered by the set $\{X_n/X_{n-1} \mid n \in \mathbb{N}\}$, where $X_0=0$, we apply the Eklof Lemma (\cite[Lemma 6.2]{GT}) in order to infer that $X \in {}^{\perp} \mathcal{D}_I$. Therefore, the monomorphism $F \rightarrow F \otimes_R Q_I$ is a special $\mathcal{T}$-preenvelope of $F$, which is flat. Using Lemma~\ref{PL00}, we infer that $\mathcal{T}$ is closed under flat covers.

	\end{proof}

	\begin{lem}\label{PL17}
		Let $R$ be an almost perfect domain. Then:
		\begin{enumerate}
				\item[(i)] Let $X$ be a Thomason subset of $\Spec(R)$ not containing 0. Then there is a set of invertible ideals $\mathcal{I}$ such that $X=\bigcup_{I \in \mathcal{I}}V(I)$\footnote{In the terminology from \cite{SS}, this means that the Gabriel topology corresponding to $X$ in the sense of \cite[Lemma 2.10]{HS} has a basis of invertible ideals, and thus is \emph{perfect}.}.
			\item[(ii)] Every 1-tilting class is closed under flat covers.
		\end{enumerate}
	\end{lem}
	\begin{proof}
			$(i)$ First, by Lemma~\ref{PP01}, there is for each $\mathfrak{m} \in \mSpec(R)$ a finitely generated ideal $I^\mathfrak{m}$, such that $V(I^\mathfrak{m})=\{\mathfrak{m}\}$. It is then enough to settle $(i)$ for the singleton Thomason sets $X=\{\mathfrak{m}\}$. Let $\mathcal{I}$ be the set of all ideals $I$ such that $V(I) \subseteq \{\mathfrak{m}\}$. It is easy to see that $\mathcal{I}$ is a filter in the lattice of all ideals of $R$ closed under multiplication of ideals.
			
			Consider the set $I_\mathfrak{m}=\{I_\mathfrak{m} \mid I \in \mathcal{I}\}$ of all localization of ideals in $\mathcal{I}$ at $\mathfrak{m}$. This set is again clearly a filter of ideals of $R_\mathfrak{m}$ closed under multiplication, and it has to contain a proper ideal. As the ring $R_\mathfrak{m}$ is a local and 1-dimensional integral domain, $\mathcal{I}_\mathfrak{m}$ is necessarily the set of all non-zero ideals of $R_\mathfrak{m}$. As $\mathcal{I}_\mathfrak{n} = \{I_\mathfrak{n} \mid I \in \mathcal{I}\}$ is equal to $\{R_\mathfrak{n}\}$ for all $\mathfrak{n} \in (\mSpec(R) \setminus \{\mathfrak{m}\})$, it follows that there is a filter basis of $\mathcal{I}$ consisting of locally projective ideals. In particular, there is $J \in \mathcal{I}$ with $J \neq R$ and $J$ flat. As $R$ is almost perfect, this implies $\pd J \leq 1$.

			Finally, by \cite[Proposition 3.2]{AJ}, any $R$-module of finite projective dimension has projective dimension at most 1. Therefore, $\pd R/J = 1$, and thus $J$ is actually projective, and hence also finitely generated. Therefore, $X=\{\mathfrak{m}\}=V(J)$, where $J$ is an invertible ideal by Lemma~\ref{PL19}.

			$(ii)$ Let $\mathcal{T}$ be a 1-tilting class in $\ModR$. By Theorem~\ref{T:HS}, there is a set of finitely generated ideals $\mathcal{I}$ such that $\mathcal{T} = \mathcal{D}_\mathcal{I}$, and by (i), we can assume that this set consists of invertible ideals. Then the rest follows from Proposition~\ref{PP18}.
	\end{proof}
			Together, this yields the following characterization:
			\begin{thm}\label{T:main}
				Let $R$ be an integral domain. Then the following conditions are equivalent:
					\begin{enumerate}
						\item[(i)] $R$ is an almost perfect domain,
						\item[(ii)] the class $\mathcal{D}$ of all divisible modules is closed under flat covers,
						\item[(iii)] all 1-tilting classes in $\ModR$ are closed under flat covers, and
						\item[(iv)] all tilting classes in $\ModR$ are closed under flat covers.
					\end{enumerate}
			\end{thm}
			\begin{proof}
					$(i) \rightarrow (iii):$ Lemma~\ref{PL17}.

					$(iii) \rightarrow (iv):$ The hypothesis implies in particular that the class $\mathcal{D}$ is closed under flat covers, which by Lemma~\ref{PL03} implies that $R$ is almost perfect, and thus 1-dimensional. Therefore, any tilting class is 1-tilting by Lemma~\ref{PL89}.

					$(iv) \rightarrow (ii)$: Trivial.

					$(ii) \rightarrow (i)$: Lemma~\ref{PL03}.
			\end{proof}
			We remark that Bazzoni proved in \cite{B2} the following related characterization: An integral domain $R$ is almost perfect if and only if the class $\mathcal{D}$ is enveloping. 
			
			Finally, we generalize Lemma~\ref{PL02} and provide a negative answer for multiplicative sets which are too large in a sense.			
			\begin{prop}\label{P:main3}
					Let $R$ be a commutative ring and $S$ a multiplicative subset of $R$ consisting of non-zero divisors. Then if the class $\mathcal{D}_S=\{M \in \ModR \mid M=sM ~\forall s \in S\}$ is closed under flat covers then $\pd RS^{-1} \leq 1$.
			\end{prop}
			\begin{proof}
					Suppose that $\mathcal{D}_S$ is closed under flat covers. Then by Corollary~\ref{PC00}, there is a $\mathcal{D}_S$-preenvelope $f: R \rightarrow L$ such that $L$ is flat. Since $L$ is flat and $S$-divisible, $L$ is a $\mathcal{D}_S$-module. Then we can prove as in Lemma~\ref{L:fre} that the natural embedding $\iota: R \rightarrow RS^{-1}$ is also a $\mathcal{D}_S$-preenvelope. It follows that also $R^{(\varkappa)} \rightarrow (RS^{-1})^{(\varkappa)}$ is a $\mathcal{D}_S$-preenvelope for any cardinal $\varkappa$, and thus $RS^{-1}$ generates the class $\mathcal{D}_S$ of all $S$-divisible modules. Consequently, \cite[Theorem 1.1]{AHT} implies that $\pd RS^{-1} \leq 1$.
			\end{proof}
			\begin{remark}The condition $\pd RS^{-1} \leq 1$ holds whenever $S$ is countable, as then $RS^{-1}$ can be filtered by the set $\{R\} \cup \{R/sR \mid s \in S\}$. On the other hand, \cite[Theorem 1.1]{AHT} says that the condition $\pd RS^{-1} \leq 1$ implies that $RS^{-1}$ is a direct sum of countably presented modules. A counterexample to Question~\ref{Q00} can be then obtained for example by taking any uncountable multiplicative subset $S$ of a valuation domain $R$, as this yields an indecomposable uncountably generated module $RS^{-1}$, and thus $\pd RS^{-1} > 1$. Finally, we remark that Proposition~\ref{P:main3} can be generalized rather straightforwardly to classes of divisibility by invertible ideals (cf. \cite[Theorem 5.4]{Hrb}.\end{remark}
\subsection{Modules divisible by a regular element}
In this last section, we prove a kind of a weak converse to Theorem~\ref{T:main} for small multiplicative sets --- if the class of modules divisible by a single regular element $s \in R$ is closed under flat covers, the ring quotient $R/sR$ has to be rather close to being perfect. 
	\begin{thm}\label{locnoeth}
			Let $R$ be a commutative ring and $s \in R$ a regular element. Then if the class $\mathcal{D}_s$ of all $s$-divisible $R$-modules is closed under flat covers then the ring $\bar{R}:=R/sR$ has the following property:
			\begin{enumerate}
					\item[\emph{(P)}] every countably presented flat $\bar{R}$-module admits a filtration by ideals of $\bar{R}$.
			\end{enumerate}
	\end{thm}
	\begin{proof}
			Put $\bar{R}=R/sR$ and fix a countably presented flat $\bar{R}$-module $\bar{F}$. By Lazard's Theorem (see \cite[Corollary 2.23]{GT}), $\bar{F}$ is a direct limit of a directed system 
			$$\bar{R}^{n_1} \xrightarrow{\bar{A}_1} \bar{R}^{n_2} \xrightarrow{\bar{A}_2} \bar{R}^{n_3} \xrightarrow{\bar{A}_3} \cdots$$
			of finitely generated free modules. Then the structure maps $\bar{A}_n$ of this system are matrices over $\bar{R}$, and thus we can lift this system to a directed system 
			$$R^{n_1} \xrightarrow{A_1} R^{n_2} \xrightarrow{A_2} R^{n_3} \xrightarrow{A_3} \cdots$$ 
			such that its limit $F$ satisfies $F \otimes_R \bar{R} \simeq \bar{F}$. Suppose towards a contradiction that $\mathcal{D}_s$ is closed under flat covers. Then by Corollary~\ref{PC00}, there is a flat $\mathcal{D}_s$-preenvelope of $F$. Lemma~\ref{L:fre} then shows that the embedding $F \rightarrow F \otimes_R Q_I$ is a $\mathcal{D}_I$-envelope, and thus in particular, a special $\mathcal{D}_I$-preenvelope by the Wakamatsu Lemma. Hence, the cokernel $C$ of the localization map $F \rightarrow F[s^{-1}]$ belongs to the class $\mathcal{A}$ of the cotorsion pair $(\mathcal{A},\mathcal{D}_s)$.

			Consider the exact sequence
			$$0 \rightarrow R \rightarrow R_s \rightarrow Y \rightarrow 0,$$
			where $Y = \bigcup_{n > 0} R/s^nR$. Applying the tensor functor $- \otimes_R F$ yields a presentation of $C$ as the direct limit of the sequence
			$$Y^{n_1} \xrightarrow{A_1 \otimes_R Y} Y^{n_2} \xrightarrow{A_2 \otimes_R Y} Y^{n_3} \xrightarrow{A_3 \otimes_R Y} \cdots$$
			Since $\Hom_R(\bar{R},Y) \simeq \bar{R}$, and the functor $\Hom_R(\bar{R},-)$ commutes with direct limits, the module $\Hom_R(\bar{R},C) = \{c \in C \mid sc=0\} \subseteq C$ is isomorphic to a direct limit of sequence
			$$\bar{R}^{n_1} \xrightarrow{\bar{A}_1} \bar{R}^{n_2} \xrightarrow{\bar{A}_2} \bar{R}^{n_3} \xrightarrow{\bar{A}_3} \cdots.$$
			Therefore, $\Hom_R(\bar{R},C) \simeq \bar{F}$, and thus $\bar{F}$ is a submodule of $C$. Since $C \in \mathcal{A}$, we have by Theorem~\ref{T:abundant} that $C$ is a direct summand of a $\{R, R/sR\}$-filtered module $M$. Let $M_\alpha, \alpha<\lambda$ be a filtration of $M$ such that $M_{\alpha+1}/M_\alpha$ is isomorphic to either $R/sR$ or $R$ for any $\alpha<\lambda$, and put $N_\alpha = M_\alpha \cap \bar{F}$. Then $\bar{F} = \bigcup_{\alpha < \lambda} N_\alpha$, and 
			$$\frac{N_{\alpha+1}}{N_\alpha} = \frac{M_{\alpha+1} \cap \bar{F}}{M_\alpha \cap \bar{F}} = \frac{M_{\alpha+1} \cap \bar{F}}{M_\alpha \cap (M_{\alpha+1} \cap \bar{F})} \simeq \frac{(M_{\alpha+1} \cap \bar{F}) + M_\alpha}{M_\alpha} \subseteq \frac{M_{\alpha+1}}{M_\alpha}.$$
			Therefore, $\bar{F}$ is a $\bar{R}$-module with a filtration $\bar{F} = \bigcup_{\alpha<\lambda} N_\alpha$, where $N_{\alpha+1}/N_\alpha$ is isomorphic to submodule of either $R$ or $R/sR$ for any $\alpha < \lambda$. Since $\bar{F}$ is $s$-torsion as an $R$-module, $N_{\alpha+1}/N_\alpha$ is zero whenever $M_{\alpha+1}/M_\alpha \simeq R$. Therefore, only submodules of $R/sR$ occur as non-zero subfactors in the filtration $\bar{F} = \bigcup_{\alpha<\lambda} N_\alpha$, and thus $\bar{F}$ is indeed filtered by ideals of the ring $R/sR$.
	\end{proof}
	\begin{thm}\label{T:main2}
		Let $R$ be a commutative ring and $s \in R$ a regular element. Then the class $\mathcal{D}_s$ of all $s$-divisible modules is not closed under flat covers if any of the following conditions holds for the quotient ring $\bar{R}=R/sR$:
\begin{enumerate}
		\item[(i)] $\bar{R}$ is noetherian and contains a regular (and non-invertible) element,
	\item[(ii)] $\bar{R}$ is a non-field integral domain with countably presented ring of quotients.
\end{enumerate}
	\end{thm}
	\begin{proof}
	In both cases, we show that the property (P) is not satisfied for $\bar{R}$. By localization, it is enough to show this in the case where $\bar{R}$ is local.
	\begin{enumerate}
	\item[(i)] Let $y$ be a regular element of $\bar{R}$. The localization $F:=\bar{R}[y^{-1}]$ is then a countably presented flat $\bar{R}$-module. Suppose that $F$ admits a filtration $F=\bigcup_{\alpha<\lambda}S_\alpha$ such that $S_{\alpha+1}/S_\alpha$ is isomorphic to an ideal $I_\alpha$ of $\bar{R}$ for all $\alpha<\lambda$. Because $s$ is regular, it is also regular on each ideal $I_\alpha$. Since $F$ is $s$-divisible, it follows that $I_\alpha = s I_\alpha$ for each $\alpha<\lambda$. But as $\bar{R}$ is local and noetherian, this would imply $I_\alpha=0$ for each $\alpha<\lambda$ by Nakayama Lemma. That is a contradiction.
	\item[(ii)] Let $\bar{F}$ be the ring of quotients of the domain $\bar{R}$. Then $\bar{F}$ is a countably presented flat $\bar{R}$-module. Towards a contradiction, suppose that $\bar{F}$ is filtered by ideals of $\bar{R}$. Since $\bar{F}$ is not projective, similar argument as in (i) shows that there is a proper non-zero ideal $I$ of $\bar{R}$ such that $I$ is divisible. If $r \in I$ is non-zero, then $r \in rI$, and thus $r=ri$ for some $i \in I$. This means that $(1-i)r=0$, a contradiction with $r \neq 0$.
	\end{enumerate}
	\end{proof}
	\begin{example}
			In this example we would like to describe one example where our technique seems to fall short. Let $k$ be a field and $S=k[x,y]_{(x,y)}$, the localization of the ring of polynomials in two variables at the maximal ideal $(x,y)$. Let $T=S/(x^2,xy)$. Then $T$ is not zero dimensional (equivalently, not artinian, or equivalently, not perfect), as $(x) \subsetneq (x,y)$ is an proper inclusion of prime ideals, but it is easy to see that every non-invertible element of $T$ is a zero-divisor. That is, if $T$ occurs as a ring isomorphic to a quotient $R/sR$ of some commutative (noetherian) domain $R$ over a regular element $s$, the closure of the class of all $s$-divisible $R$-module is out of reach of both Proposition~\ref{PP18} and Theorem~\ref{T:main2}.

			To illustrate the failure of the proof of Theorem~\ref{T:main2}, let $F=T[y^{-1}]$ be the localization of $T$ at the non-nilpotent element $y \in T$. Since $\Ann(y)=(x)$, we infer that $F \simeq k[y][y^{-1}] \simeq k(y)$ as $T$-modules. Then the flat $T$-module $F$ admits a filtration by the ideals $(x),(y)$ of the ring $T$. To see this, note that $(y)$ is isomorphic to the submodule of $F$ generated by the image of the unit, and that $(x)$ is a simple $T$-module (and thus also a simple $F$-module).
	\end{example}
	\begin{example}
			All the divisibility classes considered in this paper had a flat ring epimorphism associated (in the sense of \cite[\S IX - \S XI]{SS}), which was crucial for our arguments. In this final example, we briefly discuss the situation of the possibly simplest ideal of grade higher than 1. Let $k$ be a field, and $R=k[x,y]$. Then $I=(x,y)$ is an ideal such that $\ext_R^i(R/I,R)=0$ for $i=0,1$, and thus in particular, is not invertible. We show that the regular module $R$ (and thus also any free $R$-module) does admit a flat $\mathcal{D}_I$-preenvelope. On the other hand, there is no flat ring epimorphism associated to the class $\mathcal{D}_I$ of all $I$-divisible modules (see \cite[\S X, Exercise 3]{SS}), and thus our approach of investigating $\mathcal{D}_I$ preenvelopes of flat modules does not apply here. 
			
			Indeed, let $A=\begin{pmatrix} x \\ y \end{pmatrix}$ be a matrix inducing an inclusion $R \xhookrightarrow{A} R^2$, and denote by $A^{\oplus k}$ a diagonal-block matrix of $k$ copies of the matrix $A$ for any $k>0$. Consider the well-ordered directed system
					$$R \xhookrightarrow{A} R^2 \xhookrightarrow {A^{\oplus 2}} R^4 \xhookrightarrow {A^{\oplus 4}} \cdots R^{2^k} \xhookrightarrow{A^{\oplus 2^k}} R^{2^{k+1}} \xhookrightarrow{} \cdots,$$
			and let $M$ denote its limit. Then $M$ is a flat $R$-module, and there is an inclusion $f: R \xhookrightarrow{A} M$ induced by the first map of the system. The cokernel of $f$ admits a filtration $\bigcup_{n>0}C_n$ such that $C_{n+1}/C_n$ is isomorphic to a direct sum of $2^n$ copies of the module $\Coker(A) = R^2/(x,y)R$. Note that $(R^2/(x,y)R)^\perp = \{N \in \ModR \mid N=IN\}$, and also that $M=IM$ from the construction. Therefore, $f: R \rightarrow M$ is a special $\mathcal{D}_I$-preenvelope of $R$. (This is a special case of the more general construction of the ``Fuchs-Salce'' tilting modules from \cite[\S 4]{Hrb}).
	\end{example}
	
\end{document}